\documentclass{amsart}
\usepackage{amsfonts}
\usepackage{ stmaryrd }
\usepackage{amsthm}
\usepackage{amssymb}
\usepackage{mathrsfs}
\usepackage{epstopdf}
\usepackage{tikz}
\usepackage{tikz-cd}
\usepackage{mathtools}

\DeclarePairedDelimiter\floor{\lfloor}{\rfloor}

\def\DD{\mathcal{D}}
\def\OLKp{\OL_K(p)}
\def\Rpsi{\Rsp}

\def\glob{\mathrm{glob}}
\def\ind{\mathrm{ind}}
\def\chibar{\overline{\chi}}
\def\unr{\mathrm{unr}}
\def\split{\mathrm{split}}
\def\CM{\mathrm{CM}}
\def\DU{D_U}
\def\p{\mathfrak{p}}

\def\red{\mathrm{red}}

\def\RCM{R^{\mathrm{CM}}}
\def\Runr{R^{\mathrm{unr}}}
\def\Rsp{R^{\mathrm{split}}}
\def\Rloc{R^{\mathrm{loc}}}
\def\Rglob{R^{\mathrm{glob}}}
\def\m{\mathfrak{m}}
\def\CC{\mathscr{C}}
\def\Dind{D^{\mathrm{ind}}}
\def\Dindpsi{D^{\mathrm{ind},\psi}}
\def\Rlocind{R^{\mathrm{loc,ind}}}
\def\Rlocindpsi{R^{\mathrm{loc,ind},\psi}}
\def\Rind{R^{\mathrm{ind}}}

\def\Vbar{V_k}
\def\Ubar{U_k}
\def\varepsilonbar{\overline{\varepsilon}}

\def\T{\mathbf{T}}
\def\Zbar{\overline{\Z}}
\def\Ind{\mathrm{Ind}}

\def\SG{SG}

\def\SG{S\kern-0.02em{G}}
\def\SH{S\kern-0.02em{H}}

\def\Aut{\mathrm{Aut}}

\def\Gal{\mathrm{Gal}}

\theoremstyle{plain}

\def\PGL{\mathrm{PGL}}

\def\Spec{\mathrm{Spec}}
\def\PP{\mathfrak{P}}

\def\psibar{\overline{\psi}}

\newif\iffinalrun
\iffinalrun
\else
 \fi

\iffinalrun
  \newcommand{\need}[1]{}
  \newcommand{\mar}[1]{}
\else
  \newcommand{\need}[1]{{\tiny *** #1}}
  \newcommand{\mar}[1]{\marginpar{\raggedright\tiny Fix Me:  #1 }}\fi

\usepackage{amssymb,amsmath,amsfonts,amsthm,epsfig,amscd,stmaryrd}
\usepackage{stmaryrd}
\usepackage[all,cmtip,poly]{xy}

\newtheorem{sublemma}{Sublemma}

\newtheorem{theorem}[subsubsection]{Theorem}

\newtheorem{lemma}[subsubsection]{Lemma}

\newtheorem{cor}[subsubsection]{Corollary}

\theoremstyle{definition}
\newtheorem{df}[subsubsection]{Definition}

\theoremstyle{remark}
\newtheorem{assumption}[subsubsection]{Assumption}
\newtheorem{remark}[subsubsection]{Remark}

\newtheorem{question}[subsubsection]{Question}

\newcommand{\A}{\mathbf{A}}
\newcommand{\C}{\mathbf{C}}
\newcommand{\F}{\mathbf{F}}

\newcommand{\Q}{\mathbf{Q}}
\newcommand{\R}{\mathbf{R}}
\newcommand{\Z}{\mathbf{Z}}
\newcommand{\Fbar}{\overline{\F}}
\newcommand{\Qbar}{\overline{\Q}}
\def\rhobar{\overline{\rho}}

\DeclareMathOperator{\GL}{GL}

\def\loc{\mathrm{loc}}

\def\OL{\mathcal{O}}

\title{Vanishing Fourier Coefficients of Hecke Eigenforms}

\author[F. Calegari]{Frank Calegari}  \email{fcale@math.uchicago.edu} \address{The University of Chicago,
5734 S University Ave,
Chicago, IL 60637, USA}

\author[N.Talebizadeh Sardari]{Naser Talebizadeh  Sardari}  \email{ntalebiz@ias.edu} \address{The Institute
For Advanced Study,
1 Einstein Drive,
Princeton, New Jersey
08540 USA}

\thanks{The first author was supported in part by NSF Grants
  DMS-1701703 and DMS-2001097. The second author was supported in part by NSF Grant DMS-2015305}

\begin{document}

\begin{abstract} We prove that, for fixed level~$(N,p) = 1$ and~$p > 2$, there are only finitely many
Hecke eigenforms~$f$ of level~$\Gamma_1(N)$ and even weight with~$a_p(f) = 0$ which are not CM.
\end{abstract}

\maketitle

\section{Introduction}

Lehmer~\cite{MR21027} raised the question of whether~$\tau(n) = 0$ for any of the non-trivial Fourier coefficients
of Ramanujan's Delta function~$\Delta = q \prod_{n=1}^{\infty} ( 1 - q^{n})^{24} = \sum \tau(n) q^n$. He proved that if~$\tau(n) = 0$ for some~$n$, then necessarily~$\tau(p) = 0$ for a prime~$p | n$.
Lehmer's problem remains  open, as does the analogous  question for \emph{any} cuspidal eigenform~$f$ of level one.
If one weakens the hypothesis further and assumes only that~$f$ has level~$N$ for some~$N$ prime to~$p$,
then there are a number of ways in which~$a_p(f) = 0$, including the following:
\begin{enumerate}
\item If~$f$ is a modular form with CM arising from an imaginary quadratic field~$F/\Q$ in which~$p$ is inert, then~$a_p(f) = 0$.
\item If~$f$ is a weight two modular form arising from an elliptic curve~$E/\Q$ with good supersingular reduction at~$p$, and~$p \ge 5$, then~$a_p(f) = 0$.
\end{enumerate}

In this paper, we examine a vertical analogue of Lehmer's conjecture where~$p$ is fixed and  we
vary the weight. Our main theorem is as follows:

\begin{theorem} \label{theorem:main} Fix a prime~$p > 2$ and an integer~$(N,p) = 1$. Then there are only finitely many
non-CM Hecke eigenforms of level~$N$ and even weight with~$a_p(f) = 0$.
\end{theorem}

We shall deduce from this the following:

\begin{cor}  \label{cor:levelone} Fix a prime~$p$. There are only finitely many eigenforms of level~$1$ with~$a_p(f) = 0$.
\end{cor}

Our arguments are not effective.
  The existence of non-CM (modular) elliptic curves~$E/\Q$ which are supersingular at~$p$ shows that some exceptions must be included.
  The assumption~$p > 2$ and the assumptions on the weight
 are not intrinsic to our method, but rather reflect the absence
of certain~$R = \T$ theorems either when~$p = 2$ or when the residual representation~$\rhobar_f |_{G(\Q(\zeta_p))}$ is reducible. 
The weight condition can be weakened to requiring either that the weight~$n$ is even or~$n-1$ is not divisible by~$(p+1)/2$.

We now explain two  further motivations for considering this problem (in addition to the analogy with Lehmer's question).

\subsection{Analogies with counting Maass forms} Let~$\A^{(\infty)}_{\Q}$ denote the finite adeles, let~$G = \GL(2)/\Q$, and let~$U \subseteq G(\A^{(\infty)}_{\Q})$
denote a compact open subgroup. The problem of counting spaces of cuspidal modular forms of level~$U$ and weight~$n \ge 2$ amounts to computing the sum
\begin{equation}
\label{samesum}
\sum \dim \pi^{U}
\end{equation}
as~$\pi \subset L^2_{\mathrm{cusp}}(G(\Q) \backslash G(\A_{\Q}),\chi)$ ranges over all cuspidal automorphic representations 
(with a fixed 
suitably chosen central character~$\chi$) such that~$\pi_{\infty}$ corresponds to  a discrete series representation~$\mathcal{D}_n$
of weight~$n$. In contrast, the problem of counting spaces of algebraic Maass forms with eigenvalue~$\lambda = 1/4$ amounts to the same
 sum~(\ref{samesum}) except now where~$\pi_{\infty}$  corresponds to a particular principal series representation.
 The philosophical explanation for why  the first sum can be estimated precisely using the trace formula
  while the latter can not
 is that discrete series representations have positive measure in the Plancherel measure of
the unitary dual of~$\PGL_2(\R)$
whereas any fixed principal series does not. 
 (For this perspective on counting automorphic forms, see~\cite{Shin}.)
Finally, consider the problem of counting modular forms of weight~$n \ge 2$ and level prime to~$p$ with~$a_p = 0$ (the subject of this paper).
This  amounts to computing the same sum~(\ref{samesum}) where once more~$\pi_{\infty}$  corresponds to the discrete series~$\mathcal{D}_n$,
but now one \emph{additionally} requires that~$\pi_p$ is the  spherical representation of~$\GL_2(\Q_p)$ with given central
character and with Satake parameters~$\alpha$ and~$\beta$ satisfying~$\alpha+\beta=0$.
The obstracle in computing this sum is the same problem as for Maass forms except now the difficulties
have moved from the place~$\infty$ to the place~$p$, 
namely,  the representation~$\pi_p$ up to twist has
 zero measure in the Plancherel
measure of the unitary dual of~$\PGL_2(\Q_p)$. From an analytic point of view, these difficulties  are quite similar. Using the trace formula,
one can try to estimate~(\ref{samesum}) where now~$\pi_{\infty}$ or~$\pi_p$ respectively are now allowed to range over some class
of unitary representations of positive measure (amounting to allowing the Laplace eigenvalue~$\lambda$ or the Hecke eigenvalue~$a_p$ to vary in an~$\varepsilon$ ball around~$\lambda = 1/4$
or~$a_p = 0$ respectively) and then
try to control the error as~$\varepsilon$ becomes small. The  (upper) bounds one obtains in this way typically (see the discussion before Theorem~1 in~\cite{Duke}) have the shape~$O(V/\log V)$
where~$V$ is the trivial bound (which in the case of eigenforms with~$a_p = 0$ amounts to~$V \asymp n$). 
Thus one is led to ask whether the extra arithmetic structure present when considering the latter question
allows one to improve upon this analytic estimate (for which the answer is clearly yes).

\subsection{Analogies with a question of Greenberg}
 Consider
 an irreducible modular Galois representation
$$\rho_f:G_{\Q} \rightarrow \GL_2(\Qbar_p)$$
 of weight~$n \ge 2$. If~$f$ has CM by an imaginary quadratic field~$F/\Q$ in which~$p$ splits,
 then the restriction of~$\rho_f$ to~$G_{\Q_p}$ splits into a direct sum of characters.
 A well-known open question  (attributed to
 Ralph Greenberg~\cite{GV}) asks whether the converse holds:
 \begin{question} \label{green} Suppose that~$\rho_f$ splits
after restriction to~$G_{\Q_p}$. Does~$f$ necessarily have CM?
 \end{question}
Equivalently, is  any \emph{local} splitting of~$\rho_f$ due to a \emph{global} splitting over some finite extension of~$\Q$?
The condition that~$a_p = 0$ for a modular form of level prime to~$p$ turns out to imply (see Theorem~\ref{theorem:breuil} below) that the representation~$\rho_f$ is \emph{induced} after restriction to~$G_{\Q_p}$. If~$f$ has CM by an imaginary quadratic field~$F/\Q$ in which~$p$ is inert,
then~$\rho_f$ is globally induced from~$F$ and the restriction of~$\rho_f$ to~$G_{\Q_p}$ is induced from the unramified quadratic extension~$K/\Q_p$. 
Hence the 
 problem we are considering is the analogue of  Question~\ref{green} when \emph{split} is replaced by \emph{induced}.
The analogy here is not perfect, however, since (as noted above) non-CM elliptic curves with supersingular reduction provide
a negative answer to this question in general whereas the answer to Greenberg's question is  expected to always be positive.
Note that in addition to supersingular elliptic curves there are other isolated counterexamples, including
$$q + 4 q^3 - 8 q^4 - 5 q^5 - 22 q^6 - 11 q^9 + \ldots \in S_4(\Gamma_0(95),\Q),$$
        $$ q - 2q^{2} + 4q^{4} + 2q^{5} - 7q^{7} - 8q^{8} - 27q^{9} + \ldots \in S_4(\Gamma_0(154),\Q)$$
with~$a_2 = 0$ and~$a_3 = 0$ respectively. The listed forms are identified by their
labels~$95.4.a.a$ and~$154.4.a.b$ in the  $L$-functions and modular forms database~\cite{lmfdb}.

\subsection{Some preliminaries on Group representations} \label{prelims}

We recall some standard facts about Galois representations.
Let
$$\rho: \Gamma \rightarrow \GL_n(\Qbar_p)$$
be any continuous homomorphism of a compact group~$\Gamma$ (such as a local or global Galois group with its natural topology). 
As explained in~\cite[\S2]{Skinner}, $\rho$ takes values in~$\GL_n(E)$ for some finite extension~$E/\Q_p$ which comes
with a fixed embedding~$E \hookrightarrow \Qbar_p$, and thus also fixed embeddings~$\OL_E \hookrightarrow \Zbar_p$
and~$k_E = \OL_E/\pi_E \hookrightarrow \Fbar_p$. Since~$\Gamma$
is compact, the image of~$\rho$ preserves a lattice~$\Lambda \subset E^n$ and thus gives rise to a representation
$$\rho_{\Lambda}: \Gamma \rightarrow \Aut(\Lambda) \simeq \GL_n(\OL_E).$$
We obtain from this a corresponding residual representation
$$\rhobar_{\Lambda}: \Gamma \rightarrow \GL_n(k_E) \subset \GL_n(\Fbar_p).$$

The representation~$\rhobar_{\Lambda}$  may depend on~$\Lambda$. On the other hand,
by the Brauer--Nesbitt theorem, the semisimplification of~$\rhobar_{\Lambda}$ does not depend on any choices~\cite[\S3.2]{Bockle}.
We denote the correspond semisimple representation~$\Gamma \rightarrow \GL_n(\Fbar_p)$ by~$\rhobar$.
If~$k \subset \Fbar_p$ is any field such that~$\rhobar$ is valued in~$\GL_n(k)$, we say that~$k$ is a \emph{coefficient field} for~$\rhobar$.
By abuse of notation, we shall also let~$\rhobar$ denote the representation of~$\Gamma$ to~$\GL_n(k)$ for any coefficient field~$k$.

\medskip

We now prove a few straightfoward group theory lemmas to be used in the sequel.

\begin{lemma} \label{extends}
If~$H$ is an index~$2$ subgroup of~$G$ and~$\chi$ is a character of~$H$, then~$\Ind^{G}_{H} \chi$
is reducible
 if and only if~$\chi$ extends to a character of~$G$.

\end{lemma}

\begin{proof} If~$V = \Ind^{G}_{H} \chi$ is reducible, then~$V \simeq \chi_1 \oplus \chi_2$ for two characters~$\chi_i$ of~$G$.
Restricting to~$H$, we deduce that~$\chi_i |_{H} = \chi$ for at least one~$i$, and hence~$\chi$ extends to a character of~$G$.
Conversely, if~$\chi$ extends to a character of~$G$ (which by abuse of notation we also denote by~$\chi$), then~$V
\simeq \chi \oplus (\chi \otimes \eta_{G/H})$ where~$\eta_{G/H}$ is the the non-trivial quadratic character of~$G$ with kernel~$H$.
In particular, $V$ is reducible.
\end{proof}

\begin{lemma} \label{metacyclic}
Suppose~$G$ acts  irreducibly on a~$2$-dimensional vector space~$V$ over~$\Fbar_p$, but this action becomes
reducible after restriction to a normal subgroup~$H$ such that~$G/H$ is cyclic of order prime to~$p$.  Then   the projective image of~$G$
is dihedral, and the representation~$V$ is induced from an index two subgroup. Moreover,  either the index two subgroup
is unique and  contains~$H$,  or the projective image of~$G$ is~$D_4 \simeq (\Z/2\Z)^2$ and~$V$
is induced from exactly three index two subgroups of~$G$, precisely one of which contains~$H$.
\end{lemma}

\begin{proof}
We first claim that~$V |_{H}$ is completely decomposable. Let~$V' \subset V$ be a one-dimensional~$H$-stable
submodule. Since~$V$ is irreducible as a~$G$-module, there exists~$g \in G$ such that~$gV' \ne V'$. But since~$H$ is normal in~$G$,
it follows that~$gV'$ is also an~$H$-stable submodule of~$V$ and hence there is a decomposition~$V \simeq V' \oplus gV'$ as~$H$-modules.
We deduce that the representation~$V$ restricted to~$H$ is the direct sum of two characters, which implies that the projective representation of~$H$
associated to~$V$ has cyclic image (given by the image of the ratio of the characters)
 and has order prime to~$p$. The assumption that~$G/H$ is cyclic of order prime to~$p$ then implies that the image of~$G$ in~$\PGL_2(\Fbar_p)$
  is metacyclic of order prime to~$p$.
The finite subgroups of~$\PGL_2(\Fbar_p)$ of order prime to~$p$ consist of the groups~$\Z/r$, $D_{2r}$, $A_4$, $S_4$, and~$A_5$ (see for example~\cite[\S2.5]{Serre}). The latter three groups are not metacyclic.  If the projective image of~$G$ is cyclic, then the image~$\Gamma$ of~$G$ on~$V$ has
the property that~$\Gamma/Z(\Gamma)$ is cyclic. But~$\Gamma/Z(\Gamma)$ is cyclic only if it is trivial, which implies that~$\Gamma$ is abelian and acts reducibly on~$V$.
Thus the projective
image of~$G$ is the dihedral group~$D_{2r}$
for some~$r \ge 2$. If~$V$ is induced from more than one index two subgroup, then the projective image of~$G$ must contain at least
two cyclic subgroups of index two, and for~$D_{2r}$ and~$r \ge 2$ this happens only for~$r = 2$, when there are precisely
three such subgroups. Taken together, this proves the lemma.
\end{proof}

\begin{lemma} \label{decompose}  Let~$(A,\m)$ be a local Artinian ring with residue field~$k$.
Let~$V_A$ be an~$A[G]$-module which is free of rank~$2$ over~$A$.  
Assume that there exists a decomposition $V_A/\m=:V_k \simeq U_k \oplus U'_k$ of~$G$-modules where~$\dim U_k = \dim U'_k = 1$
and~$U_k$ is  \emph{not} isomorphic to~$U'_k$. Then any~$G$-equivariant decomposition:
$$V_A \simeq U_A \oplus U'_A$$
is unique, and ---  possibly swapping the  factors --- $U_A/\m = U_k$ and~$U'_A/\m' = U'_k$.
\end{lemma}

\begin{proof}  If~$U_A$ is any $A[G]$-module which is free of rank one as an~$A$-module,
then all the Jordan--H\"{o}lder factors of~$U_A$ are isomorphic to~$U_A/\m$. 
Assume that there exists two decompositions~$V_A \simeq U_A \oplus U'_A$ and~$V_A \simeq T_A \oplus T'_A$.
Since~$U_k$ is not isomorphic to~$U'_k$, the decomposition~$V_k = U_k \oplus U'_k$ is unique and
thus (up to reordering) we may assume that~$U_A/\m = U_k$ and also~$T_A/\m = U_k$.
If~$T_A \ne U_A$, then~$I:=(T_A + U_A) \cap U'_A$ must be non-trivial. Viewing~$I$
as a subspace of~$T_A + U_A$ (which is a quotient of~$T_A \oplus U_A$) and~$U'_A$ respectively,
we deduce that all the Jordan--H\"{o}lder factors of~$I$ are all both isomorphic to~$U_k$ and  isomorphic to~$U'_k$, a contradiction
if~$I$ is non-trivial. 
Hence~$T_A=U_A$ and similarly~$T'_A=U'_A$.
\end{proof}

\medskip

\subsection{Preliminaries on Galois deformations}
If~$k$ is a finite field, we let~$W(k)$ denote the Witt vectors of~$k$, and
$$\langle \cdot \rangle: k^{\times} \rightarrow W(k)^{\times}$$
the Teichm\"{u}ller map.
Let~$\CC$ denote the category of local Artinian~$W(k)$-algebras~$(A,\m)$ with 
a fixed identification~$A/\m = k$.

Throughout this paper, we shall often consider~$2$-dimensional representations~$V = V_k$
of a group~$\Gamma$ over a field~$k$.

\begin{df} Assume that~$V_k$ is absolutely irreducible as  a representation of~$\Gamma$.
A deformation~$V_A$ of~$V_k$ to~$A \in \CC$  consists of an~$A[\Gamma]$-module~$V_A$
 free over~$A$ such that~$V_A \otimes_{A} A/\m \simeq V_k$. Two deformations are equivalent
 if they are isomorphic as~$A[\Gamma]$-modules.
 \end{df}
 
 We now explain why this notion of deformation coincides with the alternate description in terms of matrix representations in~\cite{Mazur}.
 A fixed choice of basis for~$V_k$ gives rise to a representation: $\rhobar: \Gamma \rightarrow \GL_2(k)$.
A choice of basis for~$V_A$ lifting the given choices of basis for~$V_k$ gives rise to a representation
$$\rho: \Gamma \rightarrow \GL_2(A)$$
 such that the corresponding
residual representation is isomorphic to~$\rhobar$, and any such~$\rho$ gives rise to a module~$V_A$.
If~$\rho$ and~$\rho'$ are two  representations such that the underlying modules~$V_A$
and~$V'_A$ are isomorphic, then they are conjugate by some matrix~$M$. Since we are assuming~$V_k$ is absolutely irreducible,
it follows by Schur's lemma that this matrix must be scalar modulo~$\m$. Hence
(after scaling) we may assume~$M \in I + \m M_2(A)$, which is the usual notion of strict equivalence in~\cite{Mazur}.

The deformation functors~$D: \CC \rightarrow \mathrm{Sets}$ we consider will all be pro-represented by complete local Noetherian
rings. We shall assume basic familiarity with the further theory of Galois deformations as contained in~\cite{Mazur,Gouvea}.

\section{The Argument}
\label{argument}

\subsection{Local consequences of the condition~$a_p(f) = 0$}

Suppose we have a cuspidal modular eigenform~$f \in S_n(\Gamma_1(N),\Zbar_p)$, and let
$$\rho_f: \Gal(\Qbar/\Q) \rightarrow \GL_2(\Qbar_p)$$
denote the corresponding Galois representation. 
By the main theorem of~\cite{MR1047142,Saito}, the $p$-adic representation~$\rho_f |_{G_{\Q_p}}$
is crystalline, and the characteristic polynomial of crystalline Frobenius is~$x^2 - a_p(f)x + p^{n-1} \chi(p)$,
where~$\chi$ is the Nebentypus  character of~$f$. For~irreducible $2$-dimensional crystalline representations of~$\Gal(\Qbar_p/\Q_p)$,
the characteristic polynomial  of crystalline Frobenius is enough to determine the representation uniquely by~\cite[Prop.3.1.1]{Breuil}. 
When~$a_p(f) = 0$, there is a very simple description of the corresponding
local Galois representation which we now describe.
Let~$K/\Q_p$ denote the unique unramified quadratic extension. 
By local class field theory, there is a unique character~$K^{\times} \rightarrow G^{\mathrm{ab}}_K \rightarrow K^{\times}$
which sends~$p$ to~$1$ and~$z \in \OL^{\times}_K$ to~$z$. With respect to the two embeddings of~$K$
into~$\Qbar_p$, this gives two characters~$\varepsilon_2$, $\varepsilon'_2$ from~$G_K$
to~$\Qbar^{\times}_p$ which are permuted by the action of~$\Gal(K/\Q_p)$. 
We have the following result which follows from~\cite[Prop~3.1.2]{Breuil}:

\begin{theorem}[Breuil]  \label{theorem:breuil} Suppose that~$a_p(f) = 0$. Then
$$\left. \rho_f  \right|_{G_{\Q_p}} = \left(\Ind^{G_{\Q_p}}_{G_K} \varepsilon^{n-1}_2 \right) \otimes \psi$$
for some unramified character~$\psi$ 
 with~$\psi^2 = \chi |_{G_{\Q_p}}$. \end{theorem}

Note that~$\psi$ is \emph{a priori} only uniquely defined up to
the unramified quadratic character~$\eta_{K/\Q_p}$,
but since~$\rho_f |_{G_{\Q_p}} \otimes \eta_{K/\Q_p}
\simeq \rho_f |_{G_{\Q_p}}$, either choice is correct.

\medskip

Following~\S\ref{prelims},
  the characters~$\varepsilonbar_2$ and~$\varepsilonbar'_2$ from~$G_K$ to~$\Fbar^{\times}_p$ correspond by
local class field theory to the maps $K^{\times} \rightarrow G^{\mathrm{ab}}_K \rightarrow k^{\times}_K$ sending~$p$ to~$1$ and~$z \in \OL^{\times}_K$ to~$z \bmod p$ composed with the two embeddings of~$k_K \simeq \F_{p^2}$ into~$\Fbar_p$.
Moreover, we also have a corresponding identification
\begin{equation}
\label{resid}
\left. \rhobar_f  \right|_{G_{\Q_p}} = \left(\Ind^{G_{\Q_p}}_{G_K} \varepsilonbar^{n-1}_2 \right) \otimes \psibar
\end{equation}

For any $\rhobar_f$, we fix a coefficient field~$k = k_E \subset \Fbar_p$ as above such that~$\rhobar_f$ is valued in~$\GL_2(k)$. 
For the remainder of~\S\ref{argument} (with the exception of
of~\S\ref{pcanbeeven}) we  assume that~$p > 2$.
We now show that~$\rhobar_f$ is
absolutely irreducible  after restriction to~$G_{\Q_p}$.
 
 \begin{lemma} \label{lemma:irred} Assume that~$a_p(f) = 0$. 
 \begin{enumerate}
\item \label{part1} The representation~$\rhobar_f |_{G_{\Q_p}}$ is absolutely irreducible 
 if~$n - 1$ is not divisible by~$(p+1)$, and in particular absolutely irreducible whenever~$n$ is even.
 \item   \label{part2}  The representation~$\rhobar_f |_{G_{\Q(\zeta_p)}}$ is absolutely irreducible 
if~$n-1$ is not divisible by~$(p+1)/2$. In particular, if~$\rhobar_f |_{G_{\Q(\zeta_p)}}$ is reducible and~$n$ is even,
then~$p \equiv 1 \mod 4$ and~$(n-1) \equiv  (p+1)/2 \mod (p+1)$.
\end{enumerate}
\end{lemma}

\begin{proof}  The representation~$\rhobar_f |_{G_{\Q_p}}$ is induced from a character of an index~$2$ subgroup.
Thus, by Lemma~\ref{extends},
to prove  part~(\ref{part1}) it suffices (via the description of~$\rhobar_f |_{G_{\Q_p}}$ in
 equation~(\ref{resid}))
to show that the character~$\varepsilonbar^{n-1}_2$ of~$G_K$ does not extend to~$G_{\Q_p}$. Assume otherwise.
Then~$\varepsilonbar^{n-1}_2$ 
coincides with its~$\Gal(K/\Q_p)$ conjugate. By local class field theory, the action of~$\Gal(K/\Q_p)$ 
on~$G^{\mathrm{ab}}_K$ coincides with the action of~$\Gal(K/\Q_p)$ on~$K^{\times}$ under the Artin map. Since the non-trivial
element of~$\Gal(K/\Q_p)$ acts on the residue field~$k$ of~$K$ by Frobenius, it follows that the conjugate 
of~$\varepsilonbar_{2}$ is~$\varepsilonbar'_{2} = \varepsilonbar^{p}_2$. Thus we may assume that~$\varepsilonbar^{n-1}_2 = \varepsilonbar^{p(n-1)}_2$
and hence
$$\varepsilonbar^{(p-1)(n-1)}_2 = 1.$$
Since~$\varepsilonbar_2$ has order~$|k^{\times}| = p^2 - 1$,  this forces~$(n-1)$ to be divisible by~$(p+1)$. This proves part~(\ref{part1}).

Now suppose that~$\rhobar_f$ is irreducible as a representation of~$G=G_{\Q_p}$ but becomes reducible over~$H = G_{\Q_p(\zeta_p)}$. Since~$G/H = \Gal(\Q_p(\zeta_p)/\Q_p)$
is cyclic of order prime to~$p$, it follows by Lemma~\ref{metacyclic} that~$\rhobar_f |_{G_{\Q_p}}$ is induced from an index two subgroup of~$G_{\Q_p}$
containing~$G_{\Q_p(\zeta_p)}$. In particular, it must be induced from~$G_L$ where~$L/\Q_p$ is
the ramified quadratic extension inside~$\Q_p(\zeta_p)$.
 On the other hand, $\rhobar_f |_{G_{\Q_p}}$ is also induced from the unramified extension~$K/\Q_p$,
and so~$\rhobar |_{G_{\Q_p}}$ is irreducible and induced from at least two distinct quadratic fields. By
Lemma~\ref{metacyclic}, it follows that the projective image  of~$G_{\Q_p}$
is isomorphic to~$(\Z/2 \Z)^2$, and that the projective image restricted to~$G_K$ has order~$2$.
Since the representation~$\rhobar_f$ restricted to~$G_K$ is (up to twist) the direct sum~$\varepsilonbar^{n-1}_2 \oplus \varepsilonbar^{p(n-1)}_2$,
it follows that  the ratio of the characters~$\varepsilonbar^{n-1}_2$ and~$\varepsilonbar^{p(n-1)}_2$
is quadratic, or equivalently that
$$\varepsilonbar^{2(p-1)(n-1)}_2 = 1.$$
It follows that~$(n-1)$ is divisible by~$(p+1)/2$, which proves part~(\ref{part2}) after noting that~$(p+1)/2$ is even if~$p \equiv 3 \mod 4$.
\end{proof}

The proof of Lemma~\ref{lemma:irred} is
the only place where we use the assumption that the weight~$n$ is even. The proof otherwise proceeds
without any further changes required 
under the weaker assumption that~$(n-1)$ is not divisible by~$(p+1)$
and~$(n-1)$ is divisible by~$(p+1)/2$ only if~$p \equiv 1 \mod 4$.

\subsection{Some reductions}
By a theorem of Jochnowitz~\cite{Joch},
there  are only finitely many irreducible modular 
residual representations of level~$N$. 
By class field theory, there are only finitely many Nebentypus characters~$\chi: (\Z/N \Z)^{\times} \rightarrow  \Zbar^{\times}_p$
of level~$N$.
Moreover, if
$$\langle \chibar \rangle:  (\Z/N \Z)^{\times} \rightarrow  k^{\times} \rightarrow W(k)^{\times} \rightarrow \Zbar^{\times}_p$$
is the Teichm\"{u}ller lift of the mod~$\m_{\Zbar_p}$ reduction of~$\chi$, then~$\chi/\langle \chibar \rangle$
is valued in~$1 + \m_{\Zbar_p}$. Any such character is the square of a unique character
valued in~$1+\m_{\Zbar_p}$, because~$\sqrt{1+x} = 1+x/2 + \ldots$ converges in~$\Zbar_p$ for~$p > 2$. Thus, after a finite
global twist, we may assume that the Nebentypus characters of~$f$ are fixed and valued in~$W(k)^{\times}$.

In particular, to prove Theorem~\ref{theorem:main}, we may assume the following:

\begin{assumption} \label{assum} There exist an infinite number of cuspidal eigenforms~$f$ of level~$\Gamma_1(N)$ and even weight~$n$ with~$a_p(f) = 0$
satisfying the following further assumptions:
\begin{enumerate}
\item All such~$f$  have the same fixed residual representation
$$\rhobar=\rhobar_f:
G_{\Q} \rightarrow \GL_2(k).$$
\item There is an isomorphism 
$$\left. \rhobar  \right|_{G_{\Q_p}}  =  \left(\Ind^{G_{\Q_p}}_{G_K} \varepsilonbar^{n-1}_2 \right) \otimes \psibar$$
for fixed~$n$ and fixed~$\psibar: G_{\Q_p} \rightarrow k^{\times}$.  
\item \label{teich} If~$\psi$ denotes the character~$\langle \psibar \rangle: G_{\Q_p} \rightarrow W(k)^{\times}$, 
then the Nebentypus character~$\chi$ of~$f$ restricted to~$G_{\Q_p}$ is equal to~$\psi^2$.
By Theorem~\ref{theorem:breuil}, this is equivalent to saying that the representation~$\det(\rhobar_f) |_{G^{\mathrm{ab}}_{\Q_p}}$ evaluated
at~$p$ considered as an element of~$\Q^{\times}_p \rightarrow G^{\mathrm{ab}}_{\Q_p}$ via the Artin map
is~$\psi^2(p)$.
\end{enumerate}
\end{assumption}

\subsection{Local deformation rings}
In this section, we define some local deformation rings associated to~$\rhobar$.

Recall that we have fixed a coefficient field~$k$ for~$\rhobar$.
After increasing~$k$ if necessary, we may assume that the eigenvalue of
any element in the image of~$\rhobar$ lands in~$k$, and moreover that~$\OL_K \subseteq W(k)$, where~$W(k)$
is the Witt vectors of~$k$.

Associated to~$\rhobar |_{G_{\Q_p}}$ is a local universal deformation ring~$\Rloc$ which is a complete
local Noetherian~$W(k)$-algebra (it represents the functor~$D$ recalled below). 
We now construct a quotient of this ring corresponding to deformations which are ``induced'' from~$K$.
Let~$\Vbar$ denote the underlying representation of~$\rhobar$ over~$k$. After restricting to~$G_K$,
there is a canonical splitting~$\Vbar = \Ubar \oplus \Ubar'$ such that~$G_K$ acts on~$\Ubar$
and~$\Ubar'$ by~$\varepsilonbar^{n-1}_2 \otimes \psibar$ and~$(\varepsilonbar'_2)^{n-1} \otimes \psibar$ respectively,
and~$\Ubar$ is not isomorphic to~$\Ubar'$.

\begin{df}[Locally induced deformations]  
For a local Artinian~$W(k)$-algebra~$(A,\m) \in \CC$, let~$D^{\loc}(A)$ denote the deformations~$V_A$ of~$V_k$ to~$A$.
Let~$\Dind(A)$ denote the subset of deformations~$V_A$ which admit a splitting~$V_A = U_A \oplus U'_A$
into free~$A$-modules of rank one such that~$U_A$ and~$U'_A$ are~$G_K$-modules which are~$G_K$-deformations
of~$\Ubar$ and~$\Ubar'$
respectively. 

Let~$\Dindpsi(A)$ denote the subset  of~$\Dind(A)$  such that the action of~$p \in K^{\times} \rightarrow G^{\mathrm{ab}}_K$ on~$U_A$
is given by~$\psi(p)$, where~$\psi: G_{\Q_p} \rightarrow W(k)^{\times}$ is the
Teichm\"{u}ller lift of~$\psibar$ as in Assumption~\ref{assum}(\ref{teich}).

Finally, let~$\DU$ denote the deformations~$U_A$ of~$U_k$. \end{df}

Let~${\OLKp} = 1 + \m_K$
 denote the units in~$K^{\times}$ which are~$1 \pmod p$. Since~$p > 2$,
the exponential map gives an isomorphism~$\exp: \m_K \rightarrow {\OLKp}$.
The group~$\m_K$ may be further
identified as a topological group with~$(\Z_p)^2$ via
 an arbitrary choice of a basis for~$\m_K$ over~$\Z_p$.

\begin{lemma} \label{representingind} $\Dind$ is pro-represented by a complete~$W(k)$-algebra~$\Rlocind$  which is isomorphic
to the universal deformation ring of
$$\varepsilonbar^{n-1}_2 \otimes \psibar: G_K \rightarrow k^{\times}.$$
In particular, $\Rlocind \simeq W(k) \llbracket  {\OLKp} \oplus \Z_p \rrbracket$ is
smooth of relative dimension~$3$ over~$W(k)$.
$\Dindpsi$ is pro-represented by the quotient~$\Rlocindpsi \simeq W(k) \llbracket  {\OLKp} \rrbracket$
of~$\Rlocind$.
\end{lemma}

\begin{proof}  It suffices to identify the functors~$\Dind$ and~$\DU$.
Given an element of~$\Dind$,
there exists a corresponding free~$A$-submodule~$U_A \subset V_A$ deforming~$U_k$ by definition,
and a decomposition~$V_A = U_A \oplus U'_A$. This decomposition is unique
by Lemma~\ref{decompose}.
Conversely, given a deformation~$U_A$ of~$U_k$,
then~$V_A = \Ind^{G_{\Q_p}}_{G_K}(U_A)$ gives an element of~$\Dind(A)$, and this gives
the desired identification (one easily checks that these maps are mutual inverses).
Hence the two deformation functors coincide.

Let~$\Gamma_K$ denote the Galois group of the maximal pro-$p$ abelian extension of~$K$.
By local class field theory, we have $G^{\mathrm{ab}}_K \simeq
\widehat{K^{\times}} \simeq {\OLKp} \oplus k^{\times}_K \oplus  \widehat{\Z}$
and then (since~$p > 2$) we have an isomorphism~$\Gamma_K \simeq   {\OLKp}  \oplus \Z_p.$
It follows that the~$1$-dimensional deformation ring associated to any character~$\chibar: G_K \rightarrow k^{\times}$
is isomorphic to
$$W(k) \llbracket \Gamma_K \rrbracket \simeq
W(k) \llbracket {\OLKp} \oplus \Z_p \rrbracket,$$
 where the~$\Z_p$ factor on the right hand side is topologically generated by~$p \in K^{\times}$.
The corresponding Galois representation
$$K^{\times} \rightarrow G^{\mathrm{ab}}_K \rightarrow W(k) \llbracket  \Gamma_K \rrbracket^{\times}$$
is given explicitly as follows:
\begin{equation}
\label{explicit}
x  \mapsto   \langle \chibar(x) \rangle [\varphi(x)],
\end{equation}
where~$\varphi(x) \in \Gamma_K$ is the image of~$x$ under the natural map~$K^{\times} \rightarrow G^{\mathrm{ab}}_K \rightarrow \Gamma_K$.
 In particular,  when~$\chibar = \varepsilonbar^{n-1}_2 \psibar$, we have the desired isomorphism (noting that~$\varepsilonbar_2(p) = 1$
 and~$\psi$ is the Teichm\"{u}ller lift of~$\psibar$)
 
$$\Rlocindpsi \simeq \Rlocind/(\langle \psibar(p) \rangle [\varphi(p)] - \psi(p)) \simeq
\Rlocind/( [\varphi(p)] - 1) \simeq W(k) \llbracket  {\OLKp} \rrbracket.$$
\end{proof}

We now define local deformation rings~$R^{\loc,\unr}$ and~$R^{\loc,\split}$.

\begin{df}[Unramified and split local deformation rings]
Let~$M/\Q_p$ denote the fixed field of~$\ker(\rhobar_f |_{G_{\Q_p}})$. 
Let~$M^{\unr}$ denote the maximal unramified extension of~$M$.
We define 
subfunctors~$D^{\unr}$ and~$D^{\split}$ of~$D$ as follows:
\begin{enumerate}
\item $D^{\unr}(A) \subset D^{\loc}(A)$ consists of deformations~$V_A$ such that the action of~$G_{\Q_p}$
on~$V_A$ factors through~$\Gal(M^{\unr}/\Q_p)$.
\item $D^{\split}(A) \subset D^{\unr}(A) \subset D^{\loc}(A)$ consists of deformations~$V_A$ such that the action of~$G_{\Q_p}$
on~$V_A$ factors through~$\Gal(M/\Q_p)$.
\end{enumerate}
Let~$D^{\unr,\psi}(A)$ and~$D^{\split,\psi}(A)$
denote the subsets  of~$D^{\unr}$ and~$D^{\split}$  respectively
such  that the action of~$p \in K^{\times} \rightarrow G^{\mathrm{ab}}_K$ on~$\wedge^2 V_A$
is given by~$\psi^2(p)$. (Equivalently, the determinant character of~$V_A$ evaluated on~$p$ is~$\psi^2(p)$.)
\end{df}

\begin{lemma}  \label{unrsplit} The functor~$D^{\unr}$ is a subfunctor of~$\Dind$,
and the functor~$D^{\unr,\psi}$ is a subfunctor of~$\Dindpsi$. The functors~$D^{\unr}$, $D^{\split}$
and  $D^{\unr,\psi}$, $D^{\split,\psi}$
 are pro-representable by quotients of~$\Rlocind$ and~$\Rlocindpsi$ respectively.
 There are isomorphisms and surjections as follows:
 \begin{center}
 \begin{tikzcd}
W(k) \llbracket  {\OLKp} \oplus \Z_p \rrbracket \simeq  \Rlocind \arrow[r] \arrow[d]  & \Rlocindpsi  \simeq W(k) \llbracket  {\OLKp}  \rrbracket   \arrow[d] \\
 W(k) \llbracket \Z_p \rrbracket \simeq  R^{\loc,\unr} \arrow[r] \arrow[d]  &  R^{\loc,\unr,\psi}   \simeq W(k) \arrow[d,equal] \\
 W(k) \simeq R^{\loc,\split} \arrow[r,equals] & R^{\loc,\split,\psi}  \simeq W(k) \\
 \end{tikzcd}
 \end{center}
 where the maps in the first column
 send all elements of~${\OLKp}$ to~$1$ for the first map and all elements of~$\Z_p$ to~$1$ for
 the second.
\end{lemma}

\begin{proof}  We start by proving that~$D^{\unr}(A) \subset \Dind(A)$.
Let~$G =  \Gal(M^{\unr}/\Q_p)$, let~$H = \Gal(M^{\unr}/K)$, and let~$I  \subset H \subset G$ denote the inertia group.
There is an isomorphism~$M^{\unr} = M.\Q^{\unr}_p$. We deduce that~$I$ is cyclic of order prime to~$p$ and~$H$ is abelian.
It follows that the action of~$I$ on~$V_A$ is diagonalizable.  Since the action of~$I$ on~$V_k$ decomposes as the direct
sum~$U_k \oplus U'_k$ of distinct characters, it follows that the action of~$I$ on~$V_A$ decomposes as~$U_A \oplus U'_A$
where~$U_A$ and~$U'_A$ are free rank one~$A$-modules which reduce to~$U_k$ and~$U'_k$ modulo~$\m$ respectively.
Since~$H$ is abelian and contains~$I$, it follows that~$h U_A$ for~$h \in H$ is also a free rank one~$A$-module which is
preserved by~$I$, and thus~$V_A = h U_A \oplus h U'_A$. By Lemma~\ref{decompose},  we deduce that~$h U_A = U_A$
or~$U'_A$, and the former is ruled out by noting that~$h U_k = U_k$. Hence the decomposition~$V_A \simeq U_A \oplus U'_A$
extends to a decomposition of~$H$-modules. Since the image of~$G_K$ in~$G$ is~$H$, this says precisely
that~$V_A \in \Dind(A)$.

Having shown that~$R^{\loc,\unr}$ is a quotient of~$\Rlocind$, we can reinterpret the functors~$D^{\unr}$, $D^{\split}$ 
in terms of the deformations of the one dimensional character~$\varepsilonbar^{n-1}_2 \otimes \psibar: G_K \rightarrow k^{\times}$.
The unramified lifts correspond to deformations of the unramified character times the Teichm\"{u}ller lift of the ramified character,
and the split lift corresponds to the Teichm\"{u}ller lift of the entire representation.
Since the action of~$\Gal(K/\Q_p)$ on~$K^{\times}$ fixes~$p$, the action of~$p \in  K^{\times} \rightarrow G^{\mathrm{ab}}_K$
on~$U_A$ for any element of~$\Dind(A)$ coincides with the action of~$p$ on~$U'_A$,  and thus the subsets~$D^{\unr,\psi}(A) \subset D^{\unr}(A)$
and~$D^{\split,\psi}(A) \subset D^{\split}(A)$ are given by imposing the condition that the image of~$p$ of any deformation of our character acts by~$\psi(p)$.
The explicit descriptions of these rings can then be read off from the explicit description of $\Rind$ in Lemma~\ref{representingind} and
from equation~(\ref{explicit}).
\end{proof}

\subsection{Global deformation rings}

We now consider some global deformation rings associated to~$\rhobar: G_{\Q} \rightarrow \GL_2(k)$.

\begin{df} Let~$D(A)$ denote the deformations of~$V_k$ which are unramified outside~$Np$.
\end{df}

The functor~$D(A)$ is represented by the universal global deformation ring~$\Rglob$. 
The ring~$\Rglob$ is an~$\Rloc$-algebra by Yoneda's lemma. This allows us to define the key deformation ring~$R$ of interest,
together with auxiliary rings~$\Runr$ and~$\Rsp$.

\begin{df}  \label{splitdefinition} Let~$R = R^{\glob} \otimes_{\Rloc} \Rlocindpsi$, let~$\Runr = R^{\glob} \otimes_{\Rloc} R^{\loc,\unr}$,
and let~$\Rsp \simeq R^{\glob} \otimes_{\Rloc} R^{\loc,\split}$. 
\end{df}

The ring~$R$ represents the functor~$D^{\glob,\ind,\psi}(A)$ of deformations~$V_A$ unramified outside~$Np$ 
 such that~$V_A |_{G_{\Q_p}} \in \Dindpsi(A)$.
Similarly (following Lemma~\ref{unrsplit}) the rings~$\Runr$ and~$\Rsp$ represent the subfunctor of~$D^{\glob,\ind}(A)$ for which the corresponding
representations locally factor through~$\Gal(M^{\unr}/\Q_p)$ and~$\Gal(M/\Q_p)$ respectively,
where~$M$ is the fixed field of~$\rhobar |_{G_{\Q_p}}$.
Because~$R^{\split,\unr,\psi} \simeq R^{\split,\psi}$  by Lemma~\ref{unrsplit}, the ring~$\Rsp$ is a quotient of~$R$.
On the other hand, $\Runr$ need not
be a quotient of~$R$. 
The construction of~$R$ guarantees that all of our eigenforms (satisfying Assumption~\ref{assum})
give rise to~$\Qbar_p$-valued points of~$R$. 

The last global deformation ring we consider parametrizes deformations of~$\rhobar$ which are \emph{globally} induced from
a quadratic field~$F$.
If~$\rhobar$ is induced from~$G_F$, then there is a decomposition~$V_k \simeq U_k \oplus U'_k$ as~$G_F$-modules.
Since~$G_K \subset G_F$, this decomposition must coincide with the unique such decomposition of~$G_K$-modules for the
unramified quadratic extension~$K/\Q_p$.

\begin{df}[CM deformation rings] Suppose that~$\rhobar:G_{\Q} \rightarrow \GL_2(k)$ is induced from a quadratic field~$F/\Q$.
Let~$D^{\CM,F}(A)$ 
 denote the subset of deformations~$D(A)$ of~$V_A$  which admit a splitting~$V_A = U_A \oplus U'_A$
into free~$A$-modules of rank one such that~$U_A$ and~$U'_A$ are~$G_F$-modules which are~$G_F$-deformations
of~$\Ubar$ and~$\Ubar'$
respectively, and such that the image of~$p \in {\Q^{\times}_p} \rightarrow G^{\mathrm{ab}}_{\Q_p}$ on~$\wedge^2 V_A$
is given by~$\psi^2(p)$.
\end{df}

\begin{lemma}  \label{CM} There exist at most three quadratic fields~$F/\Q$ such that~$\rhobar:G_{\Q} \rightarrow \GL_2(k)$ is induced from~$F$.
For any such~$F$,  the functor~$D^{\CM,F}(A) \subset D(A)$ is pro-representable by
a ring~$\RCM_F$ which is a quotient of~$R$.
\end{lemma}

\begin{proof} The first claim follows from Lemma~\ref{metacyclic}. The second claim follows from the fact that any
element of~$D^{\CM,F}(A)$  gives an element of~$\Dindpsi(A)$ by restriction to~$G_K \subset G_F$.
\end{proof}

\begin{remark} \label{remark:orphan} If~$\rhobar$ is induced from a quadratic extension~$F/\Q$, the field~$F/\Q$
is either real or imaginary. There is some abuse of notation to call~$\RCM_F$
a  CM deformation ring when~$F$ is real.  This will not cause any issues since we
use only the fact that all CM deformations of~$\rhobar$
of level~$N$ lie on~$\RCM_F$ for some~$F$ (which will be imaginary).
When~$F/\Q$ is a real quadratic field,  the maximal pro-$p$ 
extension of~$F$ unramified outside~$Np$ is the compositum of the~$\Z_p$-cyclotomic extension
with a finite extension of~$F$, and so~$\RCM_F(\Qbar_p)$ consists
of a finite number of Artin representations up to twist.
\end{remark}

\subsection{The ring~$R$ is small}
We have now constructed a deformation ring~$R$  which  captures the eigenforms~$F$ with~$a_p(f) = 0$.
The first key step is establish a finiteness result for this ring.

\begin{lemma}  \label{finite} $R$ is finite as a module over~$\Rlocindpsi = W(k) \llbracket {\OLKp} \rrbracket$.
\end{lemma}

\begin{proof}  Let~$\m_{\Rlocindpsi}$ denote the maximal ideal of~$\Rlocindpsi$.
By Nakayama's lemma, to show that~$R$ is finite over~$\Rlocindpsi$ it suffices to show that~$R/\m_{\Rlocindpsi}$ is finite
over~$\Rlocindpsi/\m_{\Rlocindpsi} \simeq k$.
Now by Lemma~\ref{unrsplit}, we have isomorphisms
$$k = \Rlocindpsi/\m_{\Rlocindpsi} 
\simeq R^{\loc,\split}/\m_{R^{\loc,\split}},$$
and thus
$$R/\m_{\Rlocindpsi} 
\simeq \Rsp/\m_{R^{\loc,\split}}.$$
 If~$\Rsp$ is finite over~$R^{\loc,\split} \simeq W(k)$, then~$\Rsp/\m_{R^{\loc,\split}}$ is certainly  finite 
 over~$R^{\loc,\split}/\m_{R^{\loc,\split}} \simeq  k$, and we would be done.
 The finiteness of~$\Rsp$ follows immediately from~\cite[Thm.1(2)]{AC} under the additional Taylor--Wiles
  hypothesis that~$\rhobar|_{G_{\Q(\zeta_p)}}$
 is absolutely irreducible. Indeed, that reference proves the stronger claim that~$\Runr$ is finite over~$W(k)$,
 and~$\Rsp$ is a quotient of~$\Runr$. 
 
  It suffices to consider the remaining case when the Taylor--Wiles hypothesis fails, or
  equivalently that~$\rhobar |_{G_{\Q(\zeta_p)}}$ is reducible. This certainly implies that~$\rhobar  |_{G_{\Q_p(\zeta_p)}}$
  is reducible, and hence,
 by Lemma~\ref{lemma:irred}, we may assume that~$n-1 \equiv (p+1)/2 \mod (p+1)$ and that~$p \equiv 1 \mod 4$.
But now we may invoke Theorem~\ref{thm:append} of the appendix.
 \end{proof}

\begin{remark} An alternative approach to proving finiteness is to specialize to a height one prime ideal~$\mathfrak{p}$ of~$\Rlocindpsi$
corresponding to a representation of the form~$\varepsilon^m_1  \otimes (\Ind^{G_{\Q_p}}_{G_K} \varepsilon^{k - 1}_2) \otimes \psi$, where~$2 \le k \le p - 1$
and~$\varepsilon_1$ is the cyclotomic character.
In order for the corresponding residual representation to agree with~$ (\Ind^{G_{\Q_p}}_{G_K} \varepsilonbar^{n - 1}_2) \otimes \psibar$, it suffices
(using the identity~$\varepsilonbar_1 = \varepsilonbar^{(p+1)}_2$) 
 to chose~$m$ and~$k$
such that the  following congruence is satisfied:
$$m(p+1) + k-1 \equiv (n-1) \ \text{or} \ (n-1)p \mod (p^2 -1).$$
We may take~$k \equiv n \mod (p+1)$ unless~$n \equiv 0 \mod (p+1)$,
and we may take~$k \equiv (p+3-n) \equiv 2-n \mod (p+1)$ unless~$n \equiv 2 \mod (p+1)$, and so~$m$ and~$k$ exist as long as~$p \ge 5$.
The corresponding deformation ring~$\Rlocindpsi/\mathfrak{p}$ is a quotient of the crystalline  local deformation ring
with Hodge--Tate weights~$[m,m+k-1]$, and is thus a twist of a crystalline deformation
ring of weight~$[0,k-1]$ which is in the Fontaine--Laffaille range. Since  one expects to be able to prove~$R=\T$ theorems in this context
(exploiting the fact that the corresponding
local deformation rings are Cohen--Macaulay), this leads to explicit bounds on~$\dim_{\Q_p} (R/\mathfrak{p})[1/p]$ in terms of dimensions of spaces of modular forms of weight at most~$p-1$, although  in this approach one would also need to deal separately with the case 
when~$\rhobar |_{G(\Q(\zeta_p))}$ is reducible.
\end{remark}

We now turn to the study of the finite~$\Lambda:=
W(k) \llbracket {\OLKp} \rrbracket$-module~$R$.
The eigenforms with~$a_p(f) = 0$  captured by~$R$ give rise to a map from~$R$ to~$\Qbar_p$ and thus a prime of~$R$.
Any such prime is contained inside a minimal prime of~$R$, and since~$R$ is Noetherian, there are only
finitely many minimal prime ideals, and thus we may assume that there are infinitely many non-CM
points which lie inside a fixed minimal prime~$\PP$, or equivalently lie on a fixed irreducible component~$R/\PP$
of~$R$.

\begin{lemma}  Suppose that there are infinitely many modular Galois
representations~$\rho_f$ of level dividing~$N$ 
 giving rise to points of~$R/\PP$.
Then the support of~$R/\PP$ is all of~$\Lambda$.
\end{lemma}

\begin{proof} The ring~$R/\PP$ is a finite~$\Lambda$ module by Lemma~\ref{finite}.
The support of a finite module is closed, and thus it suffices to show that the support
includes a Zariski dense subset of~$\Lambda$.
Since there are only finitely many Galois representations of any fixed weight, we may
assume that~$R/\PP$ has points~$\rho_f$ for modular eigenforms~$f$ of infinitely many different weights.
If~$f$ has weight~$n$, then we may explicitly write down the corresponding point of~$\Lambda$.
We make the explicit identification:
$${\OLKp} \simeq (1+p) \Z_p \oplus (1 + p)^{\eta} \Z_p,$$
where~$\eta = \sqrt{u}$ for any fixed non-quadratic residue~$u$. We may then take~$X = [1+p] - 1$
and~$Y = [(1+p)^{\eta}] - 1$. The classical modular forms we are considering
all correspond
to specializations where~$z \in \OL^{\times}$ maps to~$z^{n-1}$, or equivalently to 
\begin{equation}
\label{special}
X \mapsto (1+p)^{n-1} - 1, Y \mapsto  (1 + p)^{\eta (n-1)} - 1.
\end{equation}
It suffices to show that any infinite collection of these points are Zariski dense in~$\Lambda$.
The problem is that the Zariski closure is \emph{trying to be} given by the equation
\begin{equation} \label{equation:log} H =  \eta \log(1+X) - \log(1+Y) = 0, \end{equation}
but  this is not an element of~$\Lambda \otimes \Q_p$ because the denominators
grow without bound.  (Note that we have chosen the field~$k$ so that~$\eta \in \OL_K \subseteq W(k)$.)
Alternatively, the Zariski closure wants to be~$(1+X)^{\eta} - (1+Y) = 0$,
although the corresponding formal power series (unlike~$H$)  doesn't even converge
for all~$|X|, |Y| < 1$  as we shall see shortly in Sublemma~\ref{sublemma} below.

Suppose the Zariski closure of these points is given by the vanishing set of~$F(X,Y)$ in~$W(k) \llbracket X,Y \rrbracket$.
Choose a primitive $p^m$th root of unity~$\zeta_m$ for each~$m$ and let~$\pi_m = 1 - \zeta_m$.
There is an inclusion
$$F(X,Y) \in W(k) \llbracket X,Y \rrbracket \subset W(k)[\pi_m]\llbracket X,Y \rrbracket
\subset  W(k)[\pi_m]\llbracket X/\pi_m,Y/\pi_m \rrbracket,$$
which amounts to considering the restriction of functions bounded by~$1$ on the open unit
ball~$B(1)$   to functions bounded by~$1$ on the open ball~$B(\pi_m)$. (Here~$B(r)$ denotes
the open ball centered at the origin with radius~$|r|$.)

\begin{sublemma} \label{sublemma} Suppose that~$v(\eta) = 0$ but~$\eta \notin \Z_p$. Then
\begin{equation}
H_m:=(1+X)^{\eta p^{m-1}} - (1 + Y)^{p^{m-1}}.
\end{equation}
is an element of~$W(k)[\pi_m,Y]\llbracket X/\pi_m \rrbracket \subset W(k)[\pi_m]\llbracket X/\pi_m,Y/\pi_m \rrbracket$,
but is not an element of~$W(k)[\pi_{m+1}] \llbracket X/\pi_{m+1},Y/\pi_{m+1}\rrbracket \otimes \Q_p$.
\end{sublemma}

\begin{proof} It suffices to analyze the growth of the coefficient of~$X^n$ as~$n$ varies.
The coefficient is explicitly given by a binomial coefficient, and hence its valuation is
$$v  \left(\binom{ \eta p^{m-1}}{n}\right)   = 
  v \left( \prod_{i=0}^{n-1} (\eta p^{m-1} - i) \right) -  v(n!)$$
Our assumptions on~$\eta$ imply that the valuation of~$\eta p^{m-1} - i$ is~$v(i)$ when~$v(i) \le m -1$ and~$m-1$ otherwise.
Hence the valuation of the~$X^n$ coefficient is equal to
\begin{equation}
\label{valuation}
\sum_{k=1}^{m-1} \left(1 + \floor*{\frac{n-1}{p^k}} \right)  - \sum_{k=1}^{\infty} \floor*{\frac{n}{p^k}}.
\end{equation}
A lower bound for this expression is given by
$$
\sum_{k=1}^{m-1} \floor*{\frac{n}{p^k}}  - \sum_{k=1}^{\infty}
\floor*{\frac{n}{p^k}} =
 - \sum_{k=m}^{\infty}
\floor*{\frac{n}{p^k}}  \ge  - \sum_{k=m}^{\infty}
\frac{n}{p^k}  =- \frac{n}{p^{m-1} (p-1)},$$
whereas an upper bound when~$n = p^r$ and~$r \ge m-1$ is given by
$$
\sum_{k=1}^{m-1} \left(1 + \floor*{\frac{p^r-1}{p^k}} \right)  - \sum_{k=1}^{\infty}
\floor*{\frac{p^r}{p^k}}
=
- \sum_{k=m}^{r} \frac{p^r}{p^k} 
=-\frac{n}{(p-1) p^{m-1}}
\left(1 - \frac{1}{p^{r-m-1}} \right).$$
Since~$v(\pi_m) = 1/(p^{m-1}(p-1))$, the lower bound  implies that~$H_m \in W(k)[\pi_m,Y]\llbracket X/\pi_m \rrbracket$.
On the other hand, the upper bound shows that this cannot be improved.
\end{proof}

Now~$H_m$ also vanishes at all the weights~$z \mapsto z^{n-1}$. We
claim that~$H_m$ is irreducible in~$W(k)[\pi_m]\llbracket X/\pi_m,Y/\pi_m \rrbracket$.
Viewing~$H_m$  as an element inside the larger ring~$W(k)[\pi_m]\llbracket X/\pi_1,Y/\pi_1 \rrbracket$, we \emph{do} have
the factorization
$$H_m = \prod_{i=1}^{p^{m-1}} \left(1 + Y - (1 + X)^{\eta} \zeta^{i}_{m-1} \right).$$
Certainly any factorization over the smaller ring is thus promoted to a factorization
over~$W(k)[\pi_m,Y]\llbracket X/\pi_m \rrbracket$. But if there exists a factor with~$r < p^{m-1}$
terms, then the constant term considered as a polynomial in~$(Y+1)$ will be a non-zero multiple 
of~$(1 + X)^{\eta r}$, which does not lie in this ring by Sublemma~\ref{sublemma}, a contradiction.

Because~$H_m \in  W(k)[\pi_m]\llbracket X/\pi_m,Y/\pi_m \rrbracket$ is irreducible,
we deduce that~$F(X,Y)$ must vanish at all points in~$B(\pi_m)$ where~$H_m$ vanishes.
But then~$F(X,Y)$ must vanish at the finitely many pairs~$(\xi_1 - 1,\xi_2- 1)$ of~$p$-power  roots of unity
with~$v(\xi_i-1) > v(\pi_m)$.  But repeating this with~$m$ arbitrarily large implies that~$F(X,Y)$
vanishes at all such pairs of $p$-power roots of unity, which is impossible because they are Zariski dense in~$\Lambda$.
\end{proof}

We note in passing that (formally)
$\displaystyle{\lim_{m \rightarrow \infty} \frac{H_m(X,Y)}{p^{m-1}} = H(X,Y)}$.

\begin{lemma} If a component~$R/\PP$ has infinitely many points which correspond to modular 
Galois representations which are not CM,
then the support of CM points on~$R/\PP$ is either empty or lies on a proper closed subscheme of~$\Lambda$.
\end{lemma}

\begin{proof}
If there are no CM points the result is immediate.
If~$R/\PP$ has a single point which admits CM by~$F/\Q$, then certainly~$\rhobar$ is induced from~$F$.
By Lemma~\ref{CM}, there are at most three such fields~$F$, and it suffices to prove the lemma for any given~$F$.
In particular, the points with CM by~$F$
all give rise
to points on~$\RCM_F$ and hence on the intersection~$R/\PP \otimes_R \RCM_F$.
Either this is all of~$R/\PP$, which contradicts the assumption, or, because~$R/\PP$
is irreducible, it has positive co-dimension. But since~$R$ is finite over~$\Lambda$, the dimension
of~$R$ and any of its quotients coincides with the dimension of its support, and hence we are done.
\end{proof}

To complete the proof of Theorem~\ref{theorem:main}, it suffices to show that
if the support of~$R/\PP$ is all of~$\Lambda$, then~$R/\PP$
contains a Zariski dense set of points which are CM. To this end, we use a variation
of the idea of Ghate--Vatsal~\cite{GV} to prove local indecomposability of non-CM Hida families
by specializations in weight one. 
Consider points in~$\Lambda$
corresponding to maps~${\OLKp} \rightarrow \Qbar^{\times}_p$ with finite image and such
that the ratio of any such map to its~$\Gal(K/\Q_p)$-conjugate has order greater than~$60$. 
These points are clearly Zariski dense in~$\Lambda$. The assumption that~$R/\PP$ has full support
means that for any such specialization we obtain corresponding global Galois representations:
$$\rho: G_{\Q} \rightarrow \GL_2(\Qbar_p)$$
which locally at~$p$ 
has finite image on inertia at~$p$. 
Thus, by~\cite[Thm.0.2]{PilloniStroh}
when~$\rhobar|_{G(\Q(\zeta_p))}$ is irreducible and by~\cite[Thm.1]{Sasaki} otherwise (noting that we
may assume that~$p \equiv 1 \mod 4$ and hence~$\rhobar$ is induced from a real quadratic extension), it follows that~$\rho$
is modular of weight one.
 By construction,
the image of inertia in the corresponding projective representation has order greater than~$60$ (given our choice
 of point of~$\Lambda$), and hence the global
projective representation also has order greater than~$60$.
This ensures that the projective image is not of exceptional type ($A_4$, $S_4$, or~$A_5$).
It follows that~$\rho$ must be of dihedral type. Exactly as in~\cite{GV}, all but finitely many of these forms
must additionally be of CM type. But then we have produced a Zariski
 dense set of CM points on~$R/\PP$, a contradiction, and we are done.

 \begin{remark} The fact that any infinite set of characters of the form $z \mapsto z^n$
are Zariski dense in some irreducible component of~$\Lambda = \Spec(\mathbf{Z}_p \llbracket \OL^{\times}_K \rrbracket)$ can be viewed as a special case
of a local $p$-adic analogue of Lang's       
conjecture~\cite{Lang} (See also~\cite{Serban}). The classical analogue of our example is the statement that any
infinite set of points~$(\exp(\eta x),\exp(x))$ are Zariski dense
in~$(\C^{\times})^2$ whenever~$\eta \not\in \Q$.
\end{remark}

 \subsection{Proof of Corollary~\ref{cor:levelone}}
 \label{pcanbeeven}
 If~$f$ is a form of level~$1$ which is CM, then the corresponding
 automorphic representation is induced from a character on some imaginary quadratic field~$K/\Q$.
 But then the level of~$f$ will be divisible by any prime dividing the discriminant of~$K$, a contradiction.
 Hence
 Corollary~\ref{cor:levelone} follows immediately for all~$p > 2$. For~$p = 2$, we prove directly
 that if~$a_2(f) = 0$ then~$\rho_f$ is dihedral, and so~$f$ is CM, from which the result follows by the argument above.
 Note that for~$N = 1$ and~$p = 2$ the representation~$\rhobar$ will have trivial
 semi-simplification (see~\cite[Lemme~1.7]{Chenevier}), and it follows that the image of~$\rho_f$ factors
 through the maximal pro-$2$ extension of~$\Q$ unramified outside~$2$.
 By~\cite[Prop~1.8]{Chenevier}, it follows that the image of~$\rho_f$ is isomorphic
 to the image of~$\rho_f$ restricted to inertia at~$2$. But the assumption that~$a_2(f) = 0$
 implies that~$\rho_f$ is locally induced, which now implies it is also globally induced,
 and thus CM.

\section{Acknowledgments}
This paper owes its genesis to two conversations between the authors during their respective number
theory seminars,
one in Wisconsin in February of 2018, and the second in Chicago in January of 2020. 
 The first author
would like to thank George Boxer from whom he learned the  surprising fact that the characters~$z \rightarrow z^n$
are Zariski dense in~$\mathbf{Z}_p \llbracket \OL^{\times}_K \rrbracket$, and both authors would like to thank
Patrick Allen and Toby Gee for comments.   The second author would like to thank  Professors Nigel Boston, Jordan Ellenberg, Lue Pan and Richard Taylor    for their comments on the earlier versions of this work and also he would like to thank
 the hospitality of the  department of mathematics of University of Chicago.  

 \appendix
 
 \section{Finiteness of unramified deformation rings}

Let~$\rhobar: G_{\Q} \rightarrow \GL_2(k)$
be an absolutely irreducible odd Galois representation of the form~$\Ind^{G_{\Q}}_{G_L} \chi$, 
where~$L/\Q$ is the quadratic subfield~$L \subset \Q(\zeta_p)$.
Suppose that,
up to unramified twist:
\begin{equation}
\label{type}
\left. \rhobar  \right|_{G_{\Q_p}}  = \Ind^{G_{\Q_p}}_{G_{K}} \varepsilonbar^{n-1}_2, \quad n-1 \equiv \frac{p+1}{2} \bmod p+1.
\end{equation}

The main theorem of this section is the following.

\begin{theorem} \label{thm:append}
Assume that~$p \equiv 1 \mod 4$,  so~$L/\Q$ is real.
Let~$(N,p) = 1$, and let~$\Rsp$ denote the universal
deformation ring of~$\rhobar$ consisting of representations which
are unramified outside~$N$ and totally split at~$p$ defined in Definition~\ref{splitdefinition}. Then~$\Rsp$
is finite over~$W(k)$.
\end{theorem}

Before proving this, we begin with a preliminary lemma:

\begin{lemma} \label{eandf} There exists a finite extension~$F/\Q$
with the following properties
\begin{enumerate}
\item $F$ is totally real.
\item If~$v|N$, then~$\rhobar |_{F_v}$ is trivial.
\item $\rhobar |_{G_F}$ is absolutely irreducible.
\item If~$v|p$, then~$F_v \simeq K$, the unique unramified quadratic extension of~$\Q_p$.
\end{enumerate}
\end{lemma}

\begin{proof}  The existence of~$F$
follows immediately from~\cite[Proposition 3.2]{CEven}  (see also~\cite{MB}). 
For example, one may can choose~$G=\GL_2(k)$, and then make the following choices:
\begin{enumerate}
\item $\phi_v$ for~$v|N$ is  the map~$\rhobar |_{G_{v}} \rightarrow G$,
\item $\phi_v$ for~$v = p$ is any injective map~$\Gal(K/\Q_p) \rightarrow G$,
\item $c_{\infty}$ is trivial.
\end{enumerate}
The irreducibility of~$\rhobar |_{G_F}$ is
then guaranteed by choosing~$F/\Q$ to be linearly disjoint from the fixed field~$M$ of~$\ker(\rhobar)$ 
by~\cite[Lemma 3.2]{CEven} (2).
\end{proof}

\begin{proof}[Proof of Theorem~\ref{thm:append}]
We begin with some reductions. 
If~$F/E$ is any finite extension so that~$\rhobar |_{G_F}$ remains absolutely irreducible, and~$R_F$ and~$R_E$ denote the universal deformation rings of~$\rhobar |_{G_F}$
and~$\rhobar |_{G_E}$ respectively, then the map~$R_F \rightarrow R_E$ is always finite. 
As a consequence, to prove the finiteness of~$\Rpsi$, it
suffices to replace~$\Q$ by any totally real field in which~$\rhobar$ remains absolutely irreducible.
We  replace~$\Q$ by the field~$F/\Q$ constructed in Lemma~\ref{eandf}, so that~$\rhobar |_{G_v}$ is trivial for each~$v|N$.
By Lemma~\ref{eandf}, the field~$F_v$ for~$v|p$ is precisely the unramified extension~$K/\Q_p$
for all~$v|p$. In particular,
$$\rhobar |_{F_v} \simeq \varepsilonbar^{n-1}_2 \oplus
 \varepsilonbar^{p(n-1)}_2$$
 is reducible and~$p$-distinguished. Make an arbitrary choice of one of these characters for each~$v|p$,
 which gives a distinguished choice~$U_k$ of the decomposition~$V_k = U_k \oplus U'_k$ as~$G_{F_v}$ modules
 for each~$v|p$.

We now recall the deformation ring~$R_{\DD}$ defined in~\cite[\S2]{SW}
with respect to~$\DD = (W(k),\Sigma,\emptyset)$ where~$\Sigma$ is the set of primes dividing~$N$.
Note that~$\rhobar |_{G_F}$ is absolutely irreducible and satisfies all the conditions of~\cite{SW} by construction.
The ring~$R_{\DD}$  is  global deformation ring of~$\rhobar$ unramified outside~$Np$ subject
to the following condition:  for all~$v|p$, there exists a short exact sequence
$$0 \rightarrow U_A \rightarrow  V_A \rightarrow U'_A \rightarrow 0$$
of~$G_{F_v}$-modules where~$U_A$ and~$U'_A$ are free over~$A$ of rank one and~$U_A/\m = U_k$.

After extending~$F$ if neccessary we may assume that~$[F:\Q]$ is even and thus
the hypothesis ($\mathrm{H}_{\mathrm{even}}$) of~\cite[\S3]{SW} holds.
By~\cite[Lemma~5.1.2]{Allen}, 
there exists a  lift~$\rho_0$ of~$\rhobar$ giving rise to a point on~$R_{\DD}$
  relative to our choices both over~$F$ and any totally
real extension of~$F$ in which~$\rhobar |_{G_F}$ remains irreducible. This implies that
condition ($\mathrm{H}_{\mathrm{def}}$)
of~\cite[\S3]{SW} holds for any such~$F$.
In this situation we have a corresponding Hecke ring~$\T_{\DD}$ as
defined in~\cite[p.196]{SW}, and a surjection~$R_{\DD} \rightarrow \T_{\DD}$
of~$\Lambda$-modules where~$\Lambda$ is an Iwasawa algebra~\cite[p.191]{SW} (denoted~$\Lambda_{\OL}$).
There is a surjection~$R_{\DD} \rightarrow \Rpsi$. Any deformation coming from~$\Rpsi$ 
locally has the
form~$U_A \oplus U'_A$ where the action of~$G_K$ on~$U_A$ and~$U'_A$ is via the
Teichm\"{u}ller lifts of~$\varepsilonbar^{n-1}_2$ and~$\varepsilonbar^{p(n-1)}_2$ respectively.
In particular, the map
$$\Lambda \rightarrow R_{\DD} \rightarrow \Rpsi$$
 factors through a quotient of the form~$W(k)$,
as can be seen from the formulas in~\cite{SW} on the last line of p.191 and the first line of p.192 respectively.
A more intrinsic way to see this is that the ring~$\Lambda$ represents weight space and all split representations lie in the same fixed unramified weight.
 Hence to prove that~$\Rpsi$
is finite over~$W(k)$ it suffices to show that~$R_{\DD}$ is finite over~$\Lambda$.

By taking the compositum of~$F$ with a suitably large totally subfield of a cyclotomic extension (exactly
as in the first two lines at the top of page~204 of~\cite{SW}) we may further ensure that the pair~$(F,\rho_0)$
is \emph{good} in the sense of~\cite[\S4]{SW}. It folows from~\cite[Prop~4.1]{SW} and~\cite[Prop~8.2]{SW}
that all primes of~$R_{\DD}$ are pro-modular. If~$\T_{\DD}$ is the Hecke ring defined 
in~\cite[p.196]{SW}, it follows that there is an isomorphism~$R_{\DD}/\p \simeq \T_{\DD}/\p$ for
every prime~$\p$ of~$R_{\DD}$, which implies immediately
that~$(R_{\DD})^{\red} = (\T_{\DD})^{\red}$ and hence~$(R_{\DD})^{\red}$ is finite over~$\Lambda$
(since~$\T_{\DD}$ is finite over~$\Lambda$ by~\cite[Lemma~3.3]{SW}). Since~$R_{\DD}$ is Noetherian, it follows
that~$R_{\DD}$ is also finite over~$\Lambda$.
 But~$R_{\DD}$ surjects onto~$\Rpsi/p$, and the kernel contains the image of the maximal ideal of~$\Lambda$.
It follows that~$\Rpsi/p$ is finite over~$k$ and hence~$\Rpsi$
 is also finite over~$W(k)$, as claimed.
\end{proof}

\bibliographystyle{alpha}
\bibliography{CM}

\end{document}